\documentclass[12pt]{article}
\usepackage{latexsym,amsfonts,amsthm,amsmath,amscd,amssymb}
\usepackage[dvips]{graphicx}

\newtheorem{lemma}{Lemma}[section]
\newtheorem{thm}[lemma]{Theorem}
\newtheorem{rem}[lemma]{Remark}
\newtheorem{prop}[lemma]{Proposition}
\newtheorem{cor}[lemma]{Corollary}

\newcommand\matN{{\mathbb{N}}}

\renewcommand{\hbar}{{\overline{h}}}

\newfont{\Got}{eufm10 scaled 1200}

\newcommand{\compo}{\,{\scriptstyle\circ}\,}

\newcommand{\mycap} [1] {\caption{\footnotesize{#1}}}

\newcommand\calX{{\mathcal X}}

\begin{document}

\title{Notes on the complexity of\\ $3$-valent graphs in $3$-manifolds}

\author{Ekaterina~\textsc{Pervova}\and Carlo~\textsc{Petronio}
\and Vito~\textsc{Sasso}}

\maketitle

\begin{abstract}
\noindent A theory of complexity for pairs $(M,G)$ with $M$ an
arbitrary closed $3$-manifold and $G\subset M$ a 3-valent graph was introduced
by the first two named authors, extending the original notion due to Matveev.
The complexity $c$ is known to be always additive under connected sum
away from the graphs, but not always under connected sum along (unknotted) arcs of the
graphs. In this article we prove the slightly surprising fact that
if in $M$ there is a sphere intersecting $G$ transversely at one point,
and this point belongs to an edge $e$ of $G$, then $e$ can be canceled from
$G$ without affecting the complexity. Using this fact we completely characterize
the circumstances under which complexity is additive under connected sum along graphs.
For the set of pairs $(M,K)$ with $K\subset M$ a knot, we also prove
that any function that is fully additive under connected sum along knots
is actually a function of the ambient manifold only.
\end{abstract}

The definition of complexity for a (compact) $3$-manifold was originally
given by Matveev in~\cite{MaCompl} and proved extremely fruitful
over the time, leading to an excellent topological and geometric understanding~\cite{Brunosurv}
of a large number of the ``least complicated'' such manifolds. The main technical
idea of Matveev was that to use simple spines (defined below)
rather than triangulations, which allowed him to show in
particular that complexity is additive under connected sum, and
finite-to-one on irreducible manifolds.

The first extension of the notion of complexity for pairs $(M,G)$ where $M$
is a closed $3$-manifolds and $G\subset M$ is a $3$-valent graph
(possibly with knot components) was suggested in~\cite{PeOrbCompl}
and then thoroughly investigated in~\cite{PePeJKTR}.
(In the latter paper the main emphasis and most statements were given for $G$'s
without vertices, \emph{i.e.}, for links, but all the results and constructions
extend to the graph context ---see below.)

It was shown already in~\cite{PeOrbCompl}
that the complexity function $c$ is additive under connected sum away from
the graphs, and then it was proved in~\cite{PePeJKTR} that additivity under connected
sum along the graphs holds under some restriction (explained below), that cannot be avoided.
In this paper:
\begin{itemize}
\item[(1)] We show that if an edge $e$ of a graph $G\subset M$
intersects transversely at a point a sphere that meets $G$ at that point only,
then $e$ can be canceled without affecting the complexity of $(M,G)$;
\item[(2)] Using this fact we provide a very explicit characterization
of the circumstances under which complexity is additive under connected sum along graphs;
\item[(3)] Along the same lines of reasoning we show that if a function $\psi$ on the set of pairs
$(M,K)$, with $M$ a $3$-manifold and $K\subset M$ a knot, is fully additive under connected sum along knots, then
$\psi$ is actually insensitive to knots, namely $\psi(M,K)$ depends on $M$ only.
\end{itemize}
Result (1) was first established in~\cite{VitoTesi} in the context
of an extended theory of complexity for \emph{unitrivalent} graphs in $3$-manifolds.
The extension itself proved only mildly interesting, because one easily sees
that for a unitrivalent $G$ one has $c(M,G)=c(M,G')$ with $G'$ obtained from $G$
by collapsing as long as possible and removing isolated points. But it was only within this
extended context that the ``edge-cancelation theorem'' was first guessed. With the statement
at hand, we later worked out a proof fully contained in the original trivalent context, and
we draw result (2) as a conclusion. Result (3) also follows from the same machinery,
and is actually deduced from a more abstract theorem of categorical flavour concerning
the Grothendieck group associated to the semigroup of knot-pairs.

\paragraph{Acknowledgements} The authors would like to
thank Laurent Bartholdi for the reference to the paper~\cite{Wall}.

\section{Definition of complexity\\ and review of known facts}

In this section we will rapidly recall the definitions and facts from~\cite{PeOrbCompl}
and~\cite{PePeJKTR} that we will need below. For the sake of simplicity we will only consider
\textit{orientable} (but unoriented) $3$-manifolds, even if our results hold for non-orientable
manifolds as well.

If $M$ is a closed $3$-manifold and $G\subset M$ is a trivalent graph, we call $X=(M,G)$ a \textbf{graph-pair},
and we set $M(X):=M$ and $G(X):=G$. We allow $G$ to be disconnected and to have
knot components, and we regard as vertices of $G$ only the trivalent ones
(so a knot component of $G$ does not contain vertices). We say that $X$ is a \textbf{link-pair}
if $G(X)$ is a link, \emph{i.e.}, if it does not have vertices.

For $i\in\{0,1,2,3\}$ we call \textbf{\textit{i}-sphere} in a graph-pair $X$ a sphere $S\subset  M(X)$
intersecting $G(X)$ transversely at $i$ points. And for $i\in\{0,2,3\}$ we call
\textbf{trivial \textit{i}-ball} in $X$ a ball $B\subset M(X)$ such that
$B\cap G(X)$ is empty for $i=0$, an unknotted arc for $i=2$, and an unknotted Y-graph for $i=3$.
We define a \textbf{(simple) spine}
of $X$ to be a compact polyhedron $P\subset M(X)$ such that the link of each
point of $P$ embeds in the complete graph with $4$ vertices, $P$ intersects $G(X)$ transversely,
and the complement of $P$ in $X$ is a union of trivial $i$-balls (with varying $i$). A \textbf{vertex} of $P$
is a point whose link in $P$ is the whole complete graph with $4$ vertices, and the
\textbf{complexity} $c(X)$ of $X$ is the minimal possible number of vertices for a spine of $X$.
For closed manifolds, namely for pairs $X$ with $G(X)=\emptyset$, this definition agrees
with that given by Matveev~\cite{MaCompl}, who proved the following main facts:
\begin{itemize}
\item For all $n\in\matN$ there exist finitely many
irreducible manifolds $M$ (namely such that every sphere in $M$ bounds a ball in $M$)
with $c(M)=n$; moreover if
$n>0$ then $c(M)$ is realized by spines dual to triangulations of $M$;
\item If $M_1\# M_2$ is a connected sum of $M_1$ and $M_2$ (namely
$M_1\# M_2=(M_1\setminus B_1)\cup_f(M_2\setminus B_2)$, with $B_j\subset M_j$ a ball
and $f:\partial B_1\to\partial B_2$ a homeomorphism) then
$c(M_1\# M_2)=c(M_1)+c(M_2)$.
\end{itemize}
We will not spell out here the notion of duality, but we note that these two results,
combined with the fact
that every $3$-manifold can be uniquely realized as a connected sum of irreducible ones and copies of $S^2\times S^1$,
imply that the topological meaning of the notion of manifold complexity is
very well understood at a theoretical level. In addition, very satisfactory
experimental results have been established~\cite{Brunosurv}.
Turning to graph-pairs, the following was shown in~\cite{PeOrbCompl}:
\begin{itemize}
\item For all $n\in\matN$ there exist finitely many
$(0,1,2)$-irreducible graph-pairs $X$ (namely such every $i$-sphere in $X$ bounds a trivial $i$-ball in $X$
---in particular, no $1$-sphere exists), with $c(X)=n$; moreover if
$n>0$ then $c(X)$ is realized by spines dual to triangulations of $M(X)$
that contain $G(X)$ in their $1$-skeleton.
\end{itemize}
This fact naturally prompts for an analogue of the notion of connected sum for graph-pairs,
for the quest of a unique decomposition theorem for such pairs, and for the investigation
of additivity of complexity under connected sum. This notion itself is easily introduced:
for $i\in\{0,2\}$ one says that a graph-pair $X_1\#_i X_2$ is an \textbf{\textit{i}-connected
sum} of pairs $X_1$ and $X_2$ if it is obtained by removing trivial $i$-balls from $X_1$ and
$X_2$ and gluing the resulting boundary $i$-spheres, matching the intersections
with $G(X_1)$ and $G(X_2)$ for $i=2$. Uniqueness of the expression of a graph-pair
as a connected sum of $(0,2)$-irreducible ones however does not quite hold in a literal fashion,
but weaker versions do, see~\cite{MaRoots,PeOrbDeco} and the discussion in~\cite{PePeJKTR}.
Turning to our main concern, namely
additivity under $i$-connected sum,
it was shown in~\cite{PePeJKTR} using~\cite{MaRoots} that it holds for $i=0$ and under some restrictions, but not
in general, for $i=2$:

\begin{thm}
\begin{itemize}
Let $X_1$ and $X_2$ be graph-pairs. Then:
\item $c(X_1\#_0 X_2)=c(X_1)+c(X_2)$;
\item If $X_1\#_2 X_2$ is obtained from $X_1$ and $X_2$ by removing trivial $2$-balls that
intersect edges (or knot components) $e_1$ of $X_1$ and $e_2$ of $X_2$, and if
$e_j$ does not intersect any $1$-sphere in $X_j$ for $j=1,2$, then $c(X_1\#_2 X_2)=c(X_1)+c(X_2)$;
\item There exist cases where $c(X_1\#_2 X_2)<c(X_1)+c(X_2)$.
\end{itemize}
\end{thm}
For the sake of brevity, we will describe the situation of the second item above
by saying that $X_1\#_2 X_2$ is a $2$-connected sum of $X_1$ and $X_2$ \textbf{along} $e_1$ and $e_2$.

\begin{rem}
\emph{The statement of Theorem~4.8 in~\cite{PePeJKTR} is given for link-pairs $X_1,X_2$
and asserts that $c(X_1\#_2 X_2)=c(X_1)+c(X_2)$ when neither $X_1$ nor $X_2$ contains $1$-spheres.
However one can easily check that:
\begin{itemize}
\item The absence of trivalent vertices in $G(X_j)$, namely the fact that $G(X_j)$ is a link rather than
a graph, is never used in the proof or in the results on which the proof is based;
\item For the additivity of $c$ under a $2$-connected sum $X_1\#_2 X_2$ performed
along edges $e_1$ of $G(X_1)$ and $e_2$ of $G(X_2)$, the fact that $e_j$ does not meet
$1$-spheres in $X_j$ is sufficient: in the proof one needs to show that the
$2$-sphere $S$ giving the $\#_2$ is essential, namely non-trivial and unsplittable; the former condition
is obvious, and for the latter one only needs the fact that $e_j$ does not intersect $1$-spheres in $X_j$.
\end{itemize}}
\end{rem}

The main aim of this paper is to describe exactly in which cases one has $c(X_1\#_2 X_2)=c(X_1)+c(X_2)$.
This description will be derived from the fact that edges intersecting $1$-spheres are immaterial to complexity.
Using $1$-spheres we will then also prove that for knot-pairs any invariant behaving in a fully additive
way under $2$-connected sum is actually insensitive to knots.

\section{Edges intersecting $1$-spheres\\ do not contribute to the complexity}

If $X$ is a graph-pair and $e$ is an edge or a knot component of $G=G(X)$, we denote by $X^{(e)}$ the pair
$\left(M(X),G^{(e)}\right)$ where $G^{(e)}$ is obtained from $G$ by \textit{canceling} $e$, namely:
\begin{itemize}
\item[(A)] If $e$ is an edge of $G$ ending at distinct trivalent vertices of $G$, then $G^{(e)}$ is $G$
minus the interior of $e$ (an open arc);
\item[(B)] If $e$ is a knot component of $G$, then $G^{(e)}$ is $G$ minus $e$ (a circle);
\item[(C)] If $e$ ends on both sides at one vertex $V$ of $G$, and $e'$ is the other edge
of $G$ incident to $V$, then $G^{(e)}$ is $G$ minus $V$ and the interiors of $e$ and $e'$.
\end{itemize}

\begin{thm}\label{edge:cancelation:thm}
Let $X$ be a graph-pair, suppose that $X$ contains a $1$-sphere $S$ and let
$S\cap G(X)$ be a point of an edge or knot component $e$ of $G(X)$.
Then $c(X)=c\left(X^{(e)}\right)$.
\end{thm}

\begin{proof}
Of course any spine of $X$ is also a spine of $X^{(e)}$, whence
$c\left(X^{(e)}\right)\leq c(X)$, regardless of the existence of $S$.

\medskip

Turning to the opposite inequality, suppose first that $S$ is a separating
sphere in $M$, so in particular cancelation of $e$ takes place as in case (A) above.
Cutting $X^{(e)}$ along $S$ and capping with two $0$-balls $B_1$ and $B_2$ we see that
$X^{(e)}$ is realized as some $X_1\#_0 X_2$,
where $S$ is the image in $X^{(e)}$ of the two spheres glued to perform
the $0$-connected sum, and $G^{(e)}$ is the disjoint
union of $G(X_1)$ and $G(X_2)$. Additivity of complexity under $0$-connected
sum~\cite{PePeJKTR} shows that $c\left(X^{(e)}\right)=c(X_1)+c(X_2)$.
We now claim that $c(X)\leqslant c(X_1)+c(X_2)$, which implies that
$c(X)\leqslant c\left(X^{(e)}\right)$ and hence the conclusion that
$c(X)=c\left(X^{(e)}\right)$.

To prove that $c(X)\leqslant c(X_1)+c(X_2)$ we start from spines $P_1$ and $P_2$
of $X_1$ and $X_2$ having $c(X_1)$ and $c(X_2)$ vertices,
and we construct from $P_1$ and $P_2$ a spine of $X$ without adding vertices, which
of course implies the conclusion. To do so, let us
denote by $V_1$ and $V_2$ the points of $G(X_1)$ and $G(X_2)$ at which $e$ had its ends
(before getting removed), and by $e_j$ the edge of $G(X_j)$ on which $V_j$ lies; recall that
$B_j$ is the $0$-ball in $X_j$ used to cap $X^{(e)}$ after cutting along $S$.
Then:
\begin{itemize}
\item We can assume that $e_j$ intersects $P_j$ at a surface point of $P_j$,
and that $V_j$ lies ``near'' this intersection point, as in Fig.~\ref{cancedge12:fig}-left;
    \begin{figure}
    \begin{center}
    \includegraphics[scale=0.8]{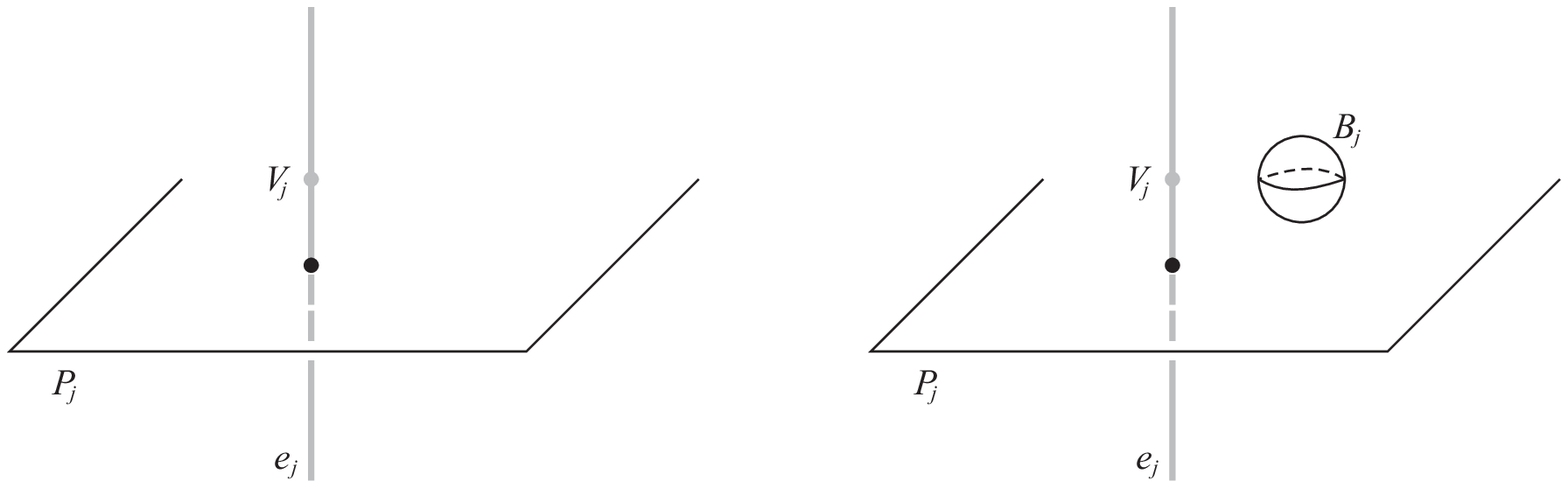}
    \mycap{A spine of one of the two pairs of which $X^{(e)}$ is the $0$-connected sum along $S$, and position of
    the ball along which the $0$-connected sum is performed}
    \label{cancedge12:fig}
    \end{center}
    \end{figure}
\item We can assume that $B_j$ lies near $V_j$ in the ball component of
$X_j\setminus P_j$ that contains $V_j$, as in Fig.~\ref{cancedge12:fig}-right;
\item Up to adding a bubble to $P_j$ we can assume
that the ball component of $X_j\setminus P_j$ containing $V_j$ and $B_j$ is a trivial $2$-ball, as in Fig.~\ref{cancedge34:fig}-left;
    \begin{figure}
    \begin{center}
    \includegraphics[scale=0.8]{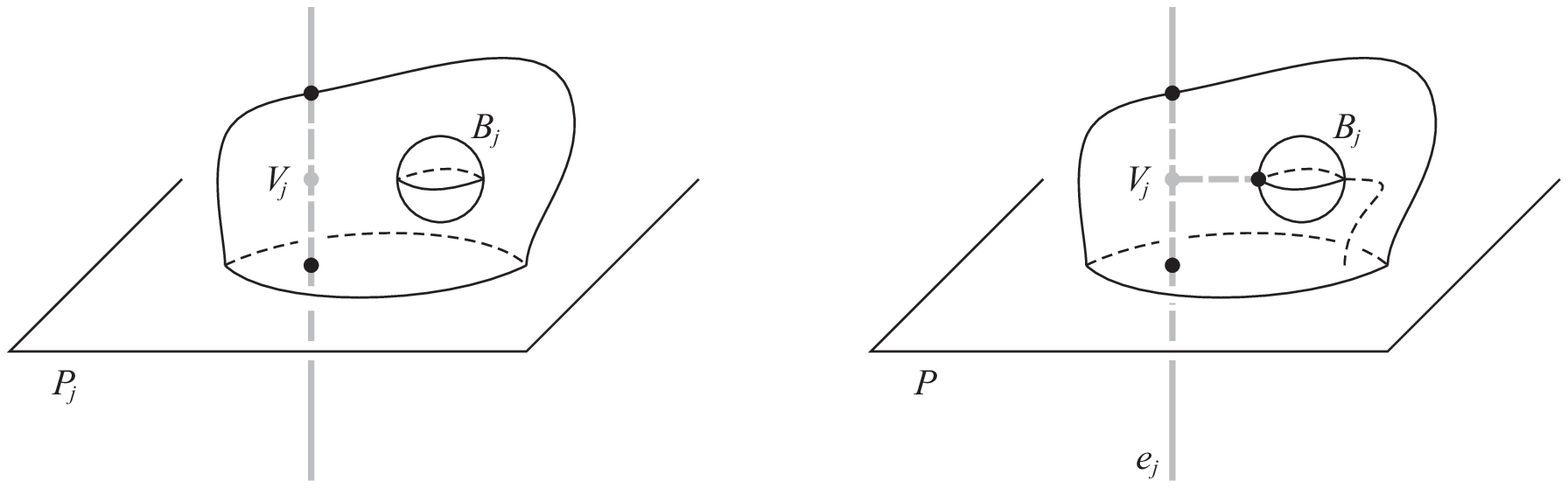}
    \mycap{Isolating the ball $B_j$ into  a $2$-ball of $X_j\setminus P_j$, and construction of a spine of $X$}
    \label{cancedge34:fig}
    \end{center}
    \end{figure}
\item Now we add to $G(X_j)$ an arc going from $V_j$ to a point $A_j$ of $\partial B_j$, and
we add to $P_j$ an arc going from $P_j$ to a different point $C_j$ of $\partial B_j$,
as in Fig.~\ref{cancedge34:fig}-right;
we then remove the interiors of $B_1$ and $B_2$, we glue $\partial B_1$ to $\partial B_2$
by a homeomorphism isotopic to that giving $X^{(e)}$ as $X_1\#_0 X_2$ and
mapping $A_1$ to $A_2$ and $C_1$ to $C_2$, and we call $P$ the union of $P_1$ and $P_2$, of
$\partial B_1=\partial B_2$, and of the two added arcs
glued along $C_1=C_2$.
\end{itemize}
This construction gives a spine $P$ of $X$ with $c(X_1)+c(X_2)$ vertices, and
the proof for the case where $S$ is separating is complete.

\medskip

We now have to deal with the case where $S$ is non-separating.
Suppose first that again the cancelation of $e$ occurs as in situation (A) above.
Then cutting $X^{(e)}$ along $S$ and capping with two balls $B_1$ and $B_2$ we get some
$Y$ such that $X^{(e)}=Y\#_0 Z$ with $Z=(S^2\times S^1,\emptyset)$.
Since $c(Z)=0$, additivity of complexity under $\#_0$ implies that
$c\left(X^{(e)}\right)=c(Y)$, therefore we get the desired conclusion as soon as we show
that $c(X)\leq c(Y)$. The construction of a spine of $X$ starting from a spine $P$
of $Y$ having $c(Y)$ vertices and without addition of vertices is actually
identical to that shown above, except that one may initially have in the first place
that $V_1$ and $V_2$ belong to the same arc of $G(Y)\setminus (P\cup V(G(Y)))$,
as in Fig.~\ref{nonsepA:fig}-left,
    \begin{figure}
    \begin{center}
    \includegraphics[scale=0.8]{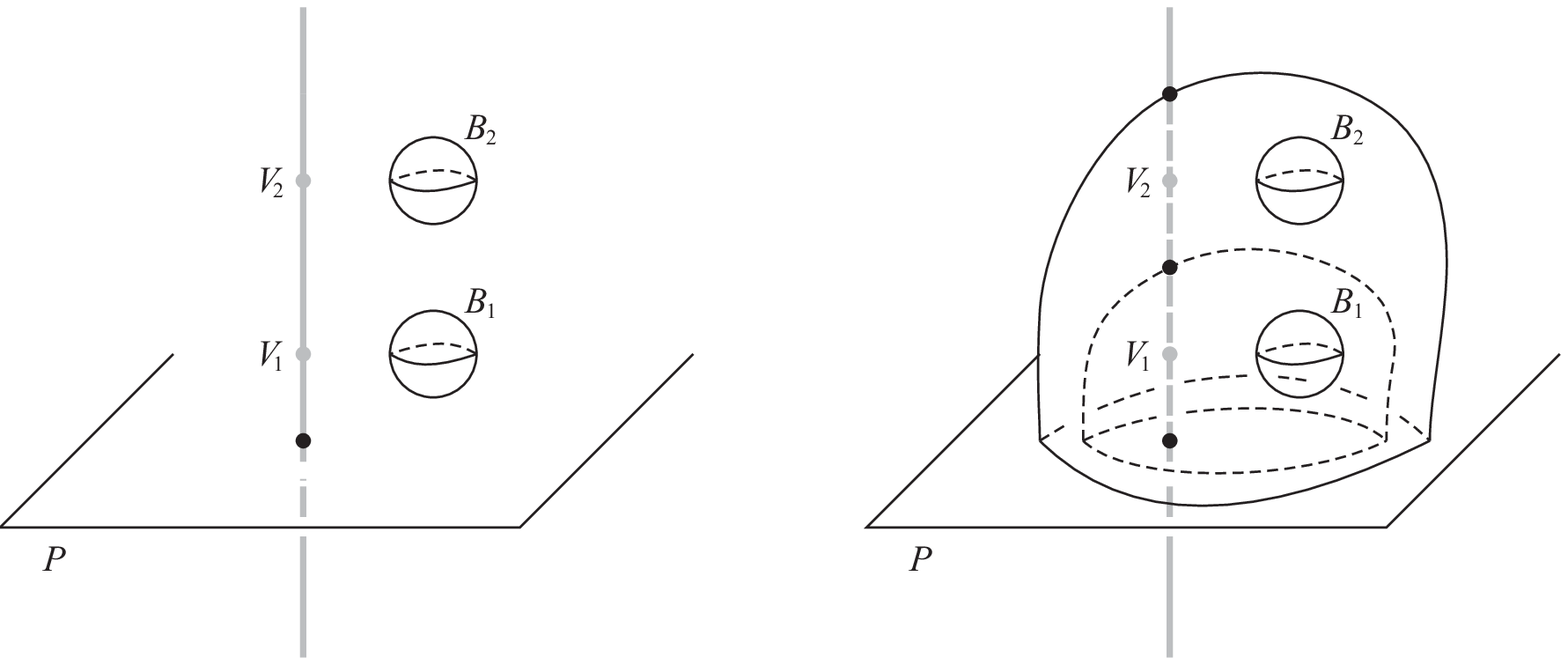}
    \mycap{After cutting along a non-separating $S$ in case (A), one can
    assume that the capping balls $B_1$ and $B_2$ are contained in different $1$-balls in $Y\setminus P$}
    \label{nonsepA:fig}
    \end{center}
    \end{figure}
in which case one performs the construction described in Fig.~\ref{nonsepA:fig}-right,
after which one can proceed precisely as before.

Turning to a non-separating $S$ in situation (B), we again cut $X^{(e)}$ along
$S$ and cap using balls $B_1$ and $B_2$, getting some $Y$ with
$X^{(e)}=Y\#_0 Z$ and $Z=(S^2\times S^1,\emptyset)$ as above.
To show that $c(X)\leq c(Y)$ and conclude we pick a point $V$ that
used to belong to $e$ before it got removed, and we assume that
$V,B_1,B_2$ lie near a spine $P$ of $Y$ with $c(Y)$ vertices,
as in Fig.~\ref{nonsepB:fig}-left.
    \begin{figure}
    \begin{center}
    \includegraphics[scale=0.8]{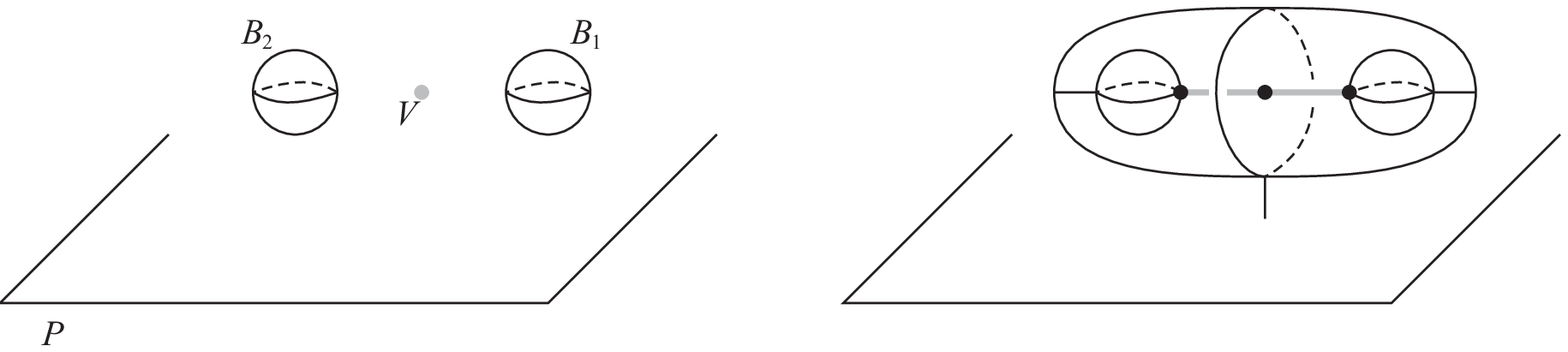}
    \mycap{Construction of a spine of $X$ from one of $Y$ in case (B)}
    \label{nonsepB:fig}
    \end{center}
    \end{figure}
We then construct a spine of $X$ as suggested in Fig.~\ref{nonsepB:fig}-right, with
the two visible spheres suitably glued together.

We are left to a non-separating $S$ as in (C). Once again $X^{(e)}=Y\#_0 Z$ and it is enough
to show that $c(X)\leqslant c(Y)$. We denote by $V$ the vertex of
$G(X)$ at which $e$ ends on both sides, and by $W$ the end different from $V$ of the edge $e'$
of $G(X)$ different from $e$ and incident to $V$. We can assume
as in Fig.~\ref{nonsepC:fig}-left
    \begin{figure}
    \begin{center}
    \includegraphics[scale=0.8]{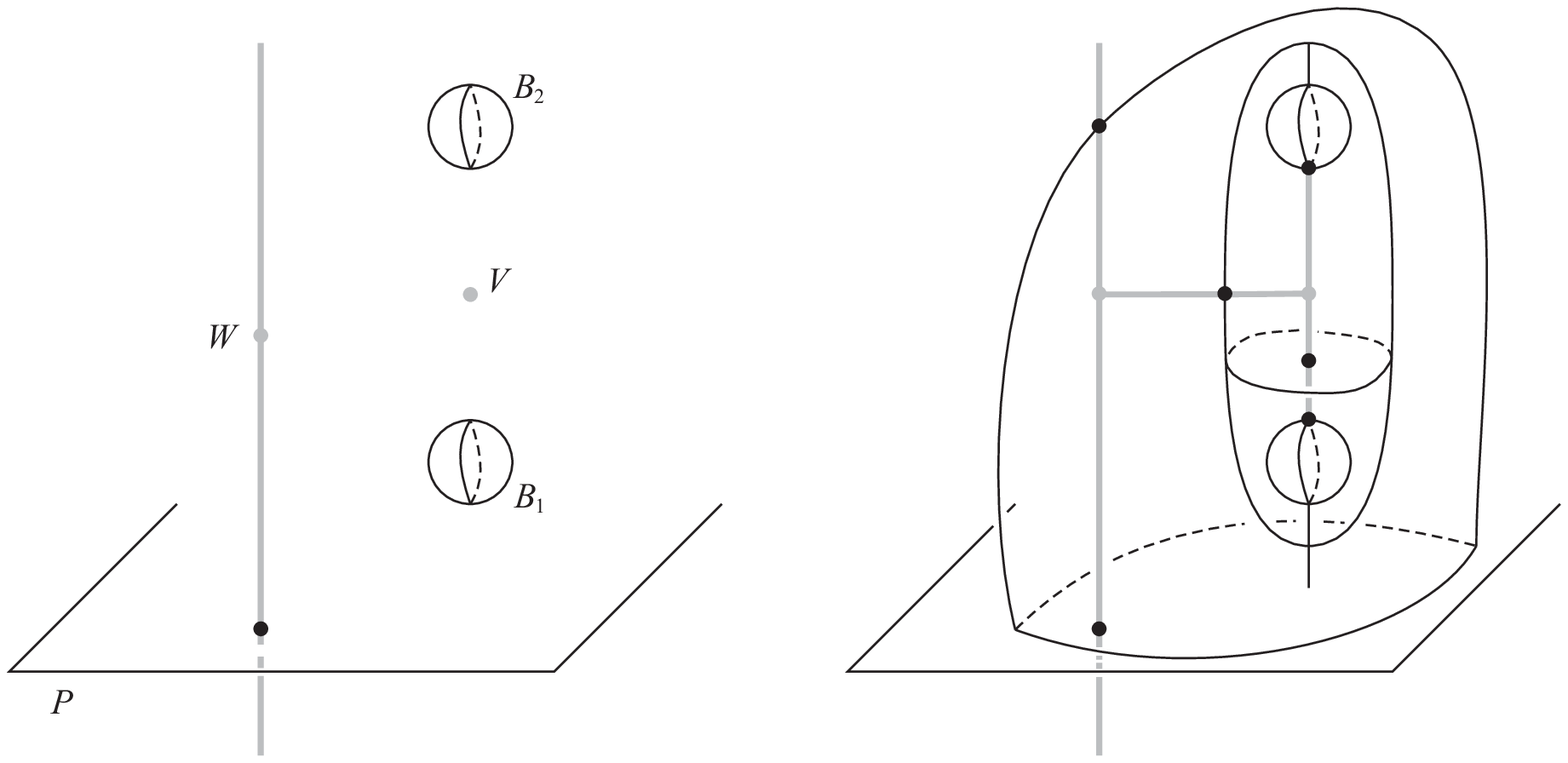}
    \mycap{Construction of a spine of $X$ from one of $Y$ in case (C)}
    \label{nonsepC:fig}
    \end{center}
    \end{figure}
that $V,B_1,B_2$ lie near $W$ in a component of $Y\setminus P$, where $P$ is a
spine of $Y$ having $c(Y)$ vertices.
Fig.~\ref{nonsepC:fig}-right now suggests how to construct a spine of $X$ without adding vertices, and
the proof is complete.
\end{proof}

\begin{rem}\label{subtle:rem}
\emph{A small subtlety in the above proof is perhaps worth pointing out.
One could imagine that, after cutting $X$ along $S$ and capping,
the two halves in which $e$ has been cut may physically appear as ``long and knotted'' arcs,
as in Fig.~\ref{subtlety:fig}-left,
    \begin{figure}
    \begin{center}
    \includegraphics[scale=0.8]{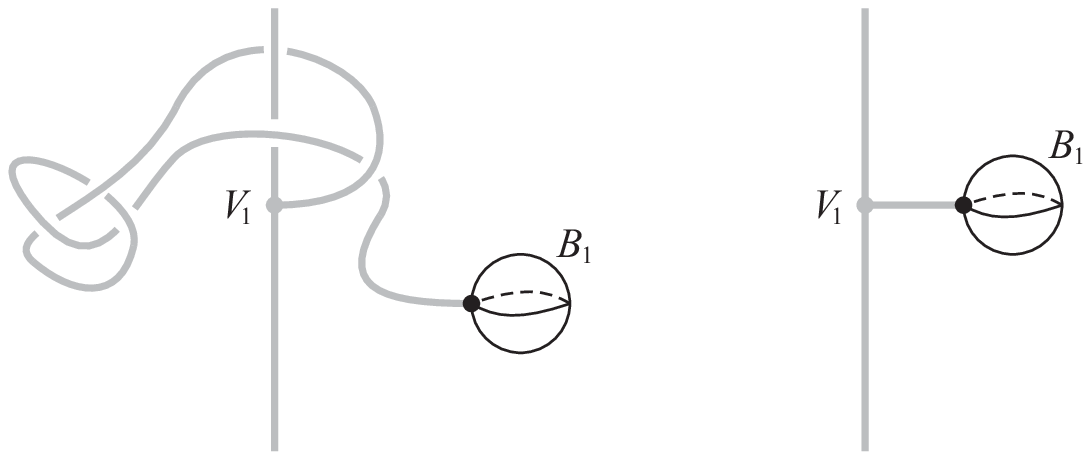}
    \mycap{An arc is always short and straight: the two configurations shown are the same}
    \label{subtlety:fig}
    \end{center}
    \end{figure}
but an arc with a free end is always ``short and straight,''
as in Fig.~\ref{subtlety:fig}-right, therefore our construction always applies. }
\end{rem}

\begin{rem}
\emph{The proof of additivity of $c$ under $0$-connected sum is carried out in~\cite{PePeJKTR}, just as the original
proof in~\cite{MaCompl}, in two steps. One first shows by hand that $c$ is subadditive,
namely that $c(X\#_0Y)\leqslant c(X)+c(Y)$ for all $X$ and $Y$, and then one shows the opposite
inequality using normal surfaces and a suitable decomposition theory: for manifolds~\cite{MaCompl},
the Haken-Kneser-Milnor decomposition of manifolds into prime ones, and, for graph-pairs~\cite{PePeJKTR},
the $(0,2)$-reduction of Matveev's root theory~\cite{MaRoots}. Analyzing the above proof one sees that
subadditivity of $c$ under $\#_0$ would not be sufficient to carry out the argument.
In other words, it is the hard part of additivity, depending on root theory,
that one needs in order to show that edges $e$ as in the statement can be canceled without affecting complexity.}
\end{rem}

\begin{rem}\label{weakly:cited:rem}
\emph{It was already shown in~\cite{PePeJKTR} that if $K$ is any knot in $S^3$ and
$D=(S^2\times S^1,\{*\}\times S^1)$ then $X=(S^3,K)\#_2 D$ has complexity $0$.
Our Theorem~\ref{edge:cancelation:thm} easily implies the result just stated,
because the knot $G(X)$ encounters some sphere $S^2\times\{*\}$ at one point only,
so the knot can be canceled from $X$ without affecting the complexity, but $X$
minus its knot is $(S^2\times S^1,\emptyset)$, that has complexity $0$.}
\end{rem}

\section{Two-connected sums under which\\
complexity is additive}

Exploiting the result from the previous section and the facts already established in~\cite{PePeJKTR}
we can now completely characterize the situations in which complexity is additive under $2$-connected sum:

\begin{thm}\label{2-add:characterization}
Let $X_1$ and $X_2$ be graph-pairs and consider a $2$-connected sum
$X_1\#_2 X_2$ performed along edges $e_1$ of $X_1$ and $e_2$ of $X_2$.
Then $c(X_1\#_2 X_2)=c(X_1)+c(X_2)$ if and only if one of the
following holds:
\begin{itemize}
\item[(1)] For $j=1,2$ the edge $e_j$ does not meet $1$-spheres in $X_j$;
\item[(2)] For $j=1,2$ one has $c\left(X_j^{(e_j)}\right)=c(X_j)$, namely $e_j$ can be canceled
from $X_j$ without affecting complexity.
\end{itemize}
\end{thm}

\begin{proof}
Suppose that $c(X_1\#_2 X_2)=c(X_1)+c(X_2)$ and that (1) does not
happen. Then up to permutation we have that $X_1$ meets a $1$-sphere
$S$ at a point of $e_1$, therefore
$c(X_1)=c\left((X_1)^{(e_1)}\right)$ by
Theorem~\ref{edge:cancelation:thm}. We can of course assume that $S$
does not meet the ball $B_1\subset X_1$ along which the
$2$-connected sum is performed.

Suppose first that $e_1$ or $e_2$
is a knot component of $G(X_1)$ or $G(X_2)$. Then the traces in $X_1\#_2 X_2$ of
$e_1$ and $e_2$ give a single edge (or knot component) $e$ in
$X_1\#_2 X_2$, and $S$ is a $1$-sphere in $X_1\#_2 X_2$ meeting $e$ at one point,
therefore Theorem~\ref{edge:cancelation:thm}
implies that $c(X_1\#_2 X_2)=c(X')$, where $X'=(X_1\#_2 X_2)^{(e)}$.
Suppose next that both $e_1$ and $e_2$ are edges ending at vertices of
$G(X_1)$ and $G(X_2)$, and note that we can take
a parallel copy $S'$ of $S$ isotoped across $B_1$, that again meets $G(X_1)$ at a point
of $e_1$, but on the opposite side with respect to $B_1$ to the
intersection point between $S$ and $e_1$. The traces in $X_1\#_2 X_2$ of $e_1$ and $e_2$
now give rise in $G(X_1\#_2 X_2)$ to two edges $e$ and $e'$. By construction
$S$ and $S'$ are $1$-spheres in $X_1\#_2 X_2$ meeting $G(X_1\#_2 X_2)$ at points
of $e$ and $e'$, therefore Theorem~\ref{edge:cancelation:thm}
implies that $c(X_1\#_2 X_2)=c(X')$, where $X'=\left((X_1\#_2 X_2)^{(e)}\right)^{(e')}$.

In both the cases just described one has that $X'=(X_1)^{(e_1)}\#_0(X_2)^{(e_2)}$,
whence $c(X')=c\left((X_1)^{(e_1)}\right)+c\left((X_2)^{(e_2)}\right)$.
But we know that $c(X')=c(X_1\#_2 X_2)$ and $c(X_1)=c\left((X_1)^{(e_1)}\right)$, and
the standing assumption is that $c(X_1\#_2 X_2)=c(X_1)+c(X_2)$, whence
$c\left((X_2)^{(e_2)}\right)=c(X_2)$, as desired.

\medskip

Turning to the opposite implication, we know from~\cite{PePeJKTR} that
the desired equality $c(X_1\#_2 X_2)=c(X_1)+c(X_2)$ holds true under condition (1).
We then assume that (2) is satisfied but (1) is not. Then again
up to permutation we can assume that $X_1$ meets a $1$-sphere $S$ at
a point of $e_1$. With precisely the same notation and arguments
as above we then have $c(X_1\#_2 X_2)=c(X')$, and
$c(X')=c\left(X_1^{(e_1)}\right)+c\left(X_2^{(e_2)}\right)$ because
$X'=X_1^{(e_1)}\#_0 X_2^{(e_2)}$. The standing assumption is now that
$c\left(X_j^{(e_j)}\right)=c(X_j)$ for $j=1,2$, and the desired equality $c(X_1\#_2 X_2)=c(X_1)+c(X_2)$
readily follows.\end{proof}

\begin{rem}
\emph{We know from Theorem~\ref{edge:cancelation:thm} that
$c\left(X^{(e)}\right)=c(X)$ when $e$ meets a $1$-sphere in $X$, but
there are many cases where $c\left(X^{(e)}\right)=c(X)$ and $e$
\emph{does not} meet $1$-spheres in $X$. For instance, if $P$ is a
minimal spine of a closed manifold $M$ not containing non-separating
spheres, for any knot $K$ dual to a region of $P$ one has
$c(M,K)=c(M)$, so $K$ can be canceled without affecting complexity,
but $K$ does not meet $1$-spheres in $(M,K)$. This shows in
particular that in Theorem~\ref{2-add:characterization} cases (1)
and (2) have a large overlapping.}
\end{rem}

\section{The Grothendieck group of knot-pairs\\ does not see the knots}

In this section we will show a further result that can be proved with the aid
of $1$-spheres in graph-pairs, as our Theorem~\ref{edge:cancelation:thm} above.
We will soon be restricting our attention to knot-pairs, namely to
pairs $(M,K)$ with $K\subset M$ a knot, but we first establish
the following perhaps not completely intuitive result, that basically follows from
the observation made in Remark~\ref{subtle:rem}, applied to case (B) of the cancelation of an edge:

\begin{prop}\label{immaterial:knot:prop}
Let $X_1$ and $X_2$ be graph-pairs such that $M(X_1)=M(X_2)$ is
connected and $G(X_j)=G_0\sqcup K_j$ with $K_j$ a knot intersected
by a $1$-sphere $S_j$ in $X_j$. Then $X_1\cong X_2$.
\end{prop}

\begin{proof}
Cutting $X_j$ along the $1$-sphere $S_j$ and capping we get some $N_j$ containing
$G_0$, two balls $B_{j,1},B_{j,2}$ and an arc $a_j$ from $\partial B_{j,1}$ to $\partial B_{j,2}$
such that $X_j$ is obtained from $(N_j,G_0\cup a_j)$ by removing the interiors
of $B_{j,1},B_{j,2}$ and gluing, matching the ends of $a_j$.
Now $N_1$ and $N_2$ can be identified to the same connected $N$ with $G_0\subset N$, and the
unions $B_{j,1}\cup a_j\cup B_{j,2}$ can be isotoped to each other in $N$ away from $G_0$,
which implies the conclusion.
\end{proof}

Turning to knot-pairs, we first refine our setting to an oriented
context, namely we consider pairs $(M,K)$ with $M$ a closed
\textit{oriented} $3$-manifold and $K\subset M$ an \textit{oriented}
knot, and we denote by $\calX_{\textrm{knot}}^{\textrm{(or)}}$ the
set of all such pairs viewed up to homeomorphisms of pairs that
preserve orientations. We next note that $\#_2$ can now be
introduced as a well-defined, binary and commutative operation on
$\calX_{\textrm{knot}}^{\textrm{(or)}}$: to define
$(M_1,K_1)\#_2(M_2,K_2)$ one must remove trivial $2$-balls $B_j$
from $(M_j,K_j)$ and glue $\partial B_1$ to $\partial B_2$ insisting
that the gluing homeomorphism should reverse the orientations
induced on $\partial B_j$ by $B_j$ and on the points $(\partial
B_j)\cap K_j$ by $B_j\cap K_j$. Since $O:=(S^3,U)$, with $U$ the
unknot, is an identity element for $\#_2$, we see that
$\calX_{\textrm{knot}}^{\textrm{(or)}}$ has a natural structure of
Abelian semigroup. We next recall that following general
fact~\cite{Groth:ref}:

\begin{prop}\label{universal:prop}
Let $S$ be an Abelian semigroup with operation $\oplus$ and identity
element $o$. Define $K(S)$ as the quotient of $S\times S$ under the
equivalence relation $\sim$, where $(a,b)\sim (c,d)$ if there exists
$u\in S$ such that $a\oplus d\oplus u=b\oplus c\oplus u$. Then the
operation $[a,b]\oplus[c,d]=[a\oplus c,b\oplus d]$ is well-defined
on $K(S)$ and turns $K(S)$ into an Abelian group with identity
element $[o,o]$. Moreover $\varphi:S\to K(S)$ given by
$\varphi(a)=[a,0]$ is a homomorphism of Abelian semigroups
satisfying the following universal property: if $G$ is an Abelian
group and $\psi:S\to G$ is a homomorphism of Abelian semigroups then
there exists a unique homomorphism of groups $\gamma:K(S)\to G$ such
that $\psi=\gamma\compo\varphi$.
\end{prop}

We can now establish the following:

\begin{thm}\label{Groth:blind:thm}
If $\varphi:\calX_{\emph{knot}}^{\emph{(or)}}\to
K(\calX_{\emph{knot}}^{\emph{(or)}})$ is the universal
homomorphism them $\varphi(M,K)$ depends on $M$ only.
\end{thm}

\begin{proof}
Let $K_1$ and $K_2$ be any two knots in $M$. Recall that
$\varphi(M,K_j)=[(M,K_j),O]$ with $O=(S^3,U)$ and $U$ the unknot.
According to the definition of
$K(\calX_{\textrm{knot}}^{\textrm{(or)}})$ and the fact that $O$ is
the identity element of $\#_2$ to show that
$[(M,K_1),O]=[(M,K_2),O]$ we only need to exhibit an element $Y$ of
$\calX_{\textrm{knot}}^{\textrm{(or)}}$ such that $(M,K_1)\#_2 Y$ is
homeomorphic $(M,K_2)\#_2Y$. For such a $Y$ we use the pair
$D=(S^2\times S^1,\{*\}\times S^1)$ already defined above, and we
remark that $(M,K_j)\#_2 D=(M\#_0(S^2\times S^1),K_j')$ where
$K_j'\subset M\#_0(S^2\times S^1)$ is a knot intersecting a
$1$-sphere. Using Proposition~\ref{immaterial:knot:prop} (and
dealing with orientation matters) we conclude that indeed
$(M,K_1)\#_2 D$ and $(M,K_2)\#_2D$ are the same in
$\calX_{\textrm{knot}}^{\textrm{(or)}}$, and the conclusion follows.
\end{proof}

\begin{cor}\label{Monly:cor}
If $\psi$ is an invariant of knot-pairs with values in an Abelian
group and $\psi$ is additive under $\#_2$ then $\psi(M,K)$ depends on $M$ only.
\end{cor}

In the context of oriented knot-pairs this corollary readily follows
from the previous theorem and the universality property of
$\varphi$. But the corollary is actually true also in an unoriented
context, by the same argument used to prove
Theorem~\ref{Groth:blind:thm}: for $K_1,K_2\subset M$ and suitably
performed $\#_2$'s one has $(M,K_1)\#_2D\cong(M,K_2)\#_2 D$ and,
using additivity of $\psi$ and simplifying $\psi(D)$, one concludes
that $\psi(M,K_1)=\psi(M,K_2)$. As a matter of fact one could also
prove the corollary by adjusting Proposition~\ref{universal:prop} to
the case of multivalued binary operations, using the notion of
\emph{semihypergroup}, obtained from that of \emph{hypergroup} (due
to Wall~\cite{Wall}) by dropping the invertibility
postulate. In this context $a\oplus b$ is a
subset of $S$. To construct $K(S)$ one first defines an element $x$ of $S$
to be a \emph{scalar} if for all $a\in S$ the set $a\oplus x$
consists of a single element; one then sets $K(S)=(S\times S)/_\sim$ where
$(a,b)\sim (c,d)$ if there exists a scalar $u\in S$ such that $a\oplus
d\oplus u=b\oplus c\oplus u$, and one shows that $K(S)$
is universal with respect to maps $\psi:S\to G$ such that $\psi$ is
constant on $a\oplus b$ for all $a,b\in S$, and $\psi$ preserves the
operations and the identity elements.
Since the pair $D$ is a scalar for the semihypergroup of unoriented knot-pairs,
the general machinery gives the unoriented version of Corollary~\ref{Monly:cor}.

\begin{rem}
\emph{We have already mentioned in Remark~\ref{weakly:cited:rem}
the result of~\cite{PePeJKTR}
that $c(X)=0$ when $X=(S^3,K)\#_2 D$, with $K\subset S^3$ any knot and
$D=(S^2\times S^1,\{*\}\times S^1)$. This was actually derived from
the homeomorphism $X\cong D$, of which the homeomorphism
$(M,K_1)\#_2 D\cong (M,K_2)\#_2 D$ used in the proof
of Theorem~\ref{Groth:blind:thm} is an extension
(take $M=S^3$, $K_1=K$ and $K_2=U$, the unknot).}
\end{rem}

\vspace{1cm}

\noindent
Dipartimento di Matematica Applicata\\
Universit\`a di Pisa\\
Via Filippo Buonarroti, 1C\\
56127 PISA -- Italy\\
\ \\
pervova@guest.dma.unipi.it\\
petronio@dm.unipi.it

\vspace{1cm}

\noindent
Dipartimento di Matematica\\
Universit\`a di Roma ``Tor Vergata''\\
Via della Ricerca Scientifica, 1\\
00133 ROMA -- Italy\\
\ \\
sasso@mat.uniroma2.it

\noindent


\begin{thebibliography}{99}

\bibitem{Groth:ref}
\textsc{M.~Atiyah}, ``K-theory,'' (Notes taken by
D.~W.~Anderson, Fall 1964), W.~A. Benjamin Inc.,
New York, 1967.

\bibitem{HHMP}
\textsc{D.~Heard --- C.~Hodgson --- B.~Martelli --- C.~Petronio},
\textit{Hyperbolic graphs of small complexity}, Experiment. Math. \textbf{19} (2010), 211-236.

\bibitem{MaRoots}
\textsc{C.~Hog-Angeloni --- S.~Matveev}, \textit{Roots in
$3$-manifold topology.}, In: ``The Zieschang Gedenkschrift'', Geom.
Topol. Monogr. Vol. 14, Geom. Topol. Publ., Coventry, 2008, 295-319.


\bibitem{Brunosurv}
\textsc{B.~Martelli}, \emph{Complexity of $3$-manifolds},
In: ``Spaces of Kleinian Groups'', London Math. Soc. Lecture Note Ser., Vol. 329,
Cambridge Univ. Press, Cambridge, 2006, 91-120.


\bibitem{MaCompl}
\textsc{S.~Matveev}, \textit{Complexity theory of three-dimensional
manifolds}, Acta Appl. Math. \textbf{19} (1990), 101-130.



\bibitem{PePeJKTR}
\textsc{E.~Pervova --- C.~Petronio},
\textit{Complexity of links in $3$-manifolds},
J. Knot Theory Ramifications \textbf{18} (2009), 1439-1458.

\bibitem{PeOrbDeco}
\textsc{C.~Petronio}, \textit{Spherical splitting of $3$-orbifolds},
Math. Proc. Cambdridge Phylos. Soc \textbf{142} (2007), 269-287.

\bibitem{PeOrbCompl}
\textsc{C.~Petronio}, \textit{Complexity of $3$-orbifolds}, Topology
Appl. \textbf{153} (2006), 1658-1681.

\bibitem{VitoTesi}
\textsc{V.~Sasso}, ``Complexity of unitrivalent
graph-pairs and knots in 3-manifolds'',
PhD Thesis, Universit\`a di Roma ``Tor Vergata'', 2011.

\bibitem{Wall}
\textsc{H.~S.~Wall}, \emph{Hypergroups}, Amer. J. Math. \textbf{59} (1937), 77-98.

\end{thebibliography}
\end{document}